\documentclass{amsart}

\usepackage{amssymb,amsmath,color,ytableau,graphicx,float}
\usepackage{amsthm}
\usepackage{url}
\usepackage{tikz}

\definecolor{red}{rgb}{1,0,0}
\definecolor{blue}{rgb}{.2,.2,.8}

\def\mex{\textrm{mex}}

\newtheorem{theorem}{Theorem}[section]

\theoremstyle{definition}

\begin{document}

\title[Combinatorial proofs of Merca's  theorems]{Combinatorial proofs of Merca's identities involving the sum of different parts congruent to $r$ modulo $m$ in all partitions of $n$}
\author{Cristina Ballantine}\address{Department of Mathematics and Computer Science\\ College of the Holy Cross \\ Worcester, MA 01610, USA \\} 
\email{cballant@holycross.edu}

\maketitle

\begin{abstract} We give combinatorial proofs of several recent results due to Merca on the sum of different parts congruent to $r$ modulo $m$ in all partitions of $n$. The proofs make use of some well known involutions from the literature and some new involutions introduced here. 
\end{abstract}

{\bf Keywords:} partitions, bijections, involutions, recurrences.

{\bf MSC 2020:} 11P81, 11P84, 05A17, 05A19

\section{Introduction}

A partition $\lambda$ of $n$ is a non-increasing sequence $\lambda= (\lambda_1, \lambda_2, \ldots, \lambda_\ell)$ of positive integers that add up to $n$. We refer to the integers $\lambda_i$ as the parts of $\lambda$.  As usual, we denote by $p(n)$  the number of  partitions of $n$. Note that $p(x)=0$ if $x$ is not a non-negative integer, and since the empty partition $\emptyset$ is the only partition of $0$, we have that $p(0)=1$.

Let $m,n$, and $r$ be nonnegative integers such that $0\leq r<m$. Denote by $a_{r,m}(n)$ the sum of all different parts congruent to $r$ modulo $m$ in all partitions of $n$. Thus, for a partition $\lambda$ of $n$, a part $mj+r$ of $\lambda$ contributes $mj+r$ to $a_{r,m}(n)$ regardless of its multiplicity.  We denote by $s_m(n)$ the sum of different parts that appear at least $m$ times in partitions of $n$.  Thus, for a partition $\lambda$ of $n$, a part $a$ of $\lambda$  that occurs at least $m$ times contributes $a$ to $s_m(n)$ regardless of its multiplicity.

Recently Merca \cite{M23} proved several results relating $a_{m,r}(n)$, $s_m(n)$, and numbers of restricted partitions. In this article, we give combinatorial proofs of several  results in \cite{M23}. Combinatorial proofs of \cite[Theorem 1.3 and Corollary 1.4]{M23} are given in \cite{MS}. In \cite{M23} the author gives  a combinatorial proof of Theorem 1.6. . 
We prove combinatorially Theorem 1.3 and Corollaries 4.2, 4.4, 4.6, 4.7(i), 4.9, 4.10, 5.2, 5.3, 6.3, 6.3, 7.2, 7.3 of \cite{M23}. The corollaries are limiting cases of inequalities obtained in \cite{M23} by truncating theta series.

\section{Combinatorial proofs of theorems of Merca}

We first introduce some notation. 
We denote by $\mathcal P(n)$ the set of partitions of $n$ and by $\mathcal P$ the set of all partitions. We use this convention is used for all other sets of partitions: if $\mathcal A(n)$ denotes a set of partitions of $n$, then $$\mathcal A=\bigcup_{n\geq0}\mathcal A(n).$$ We write $\lambda\vdash n$ to mean that $\lambda$ is a partition of $n$. We also write $|\lambda|=n$ to mean that the parts of $\lambda$ add up to $n$. The number of parts of $\lambda$ is called the length of $\lambda$ and is denoted by $\ell(\lambda)$. If $\lambda^{(1)}, \lambda^{(2)}, \ldots, \lambda^{(k)}$ are partitions, we  write $(\lambda^{(1)}, \lambda^{(2)}, \ldots, \lambda^{(k)})\vdash n$ to mean $|\lambda^{(1)}|+ |\lambda^{(2)}|+ \cdots +|\lambda^{(k)}|=n$. 

We define $p_{e,o}(n):=p_e(n)-p_o(n),$ where $p_e(n)$ (respectively $p_o(n)$) is the number of partitions of $n$ with an even (respectively odd) number of parts. 

An overpartition of $n$ is a partition of $n$ in which the first occurrence of a part may be overlined. We denote by $\overline{\mathcal P}(n)$ the set of overpartitions of $n$. As in the case of partitions, we define $\overline p_{e,o}(n):=\overline p_e(n)-\overline p_o(n),$ where $\overline p_e(n)$ (respectively $\overline p_o(n)$) is the number of overpartitions of $n$ with an even (respectively odd) number of parts. 

We denote by $\mathcal Q(n)$ the set of partitions of $n$ into distinct parts and write $q(n)$ for $|\mathcal Q(n)|$. Similarty, we denote by $q_{odd}(n)$ (respectively $q_{even}(n)$) the number of partitions of $n$ into distinct parts all odd (respectively even). 

\begin{theorem}{\cite[Theorem 1.3]{M23}} \label{T1.3}Let $m,n$, and $r$ be nonnegative integers such that $0\leq r<m$. We have 
\begin{itemize}
\item[(i)] $\displaystyle a_{r,m}(n)=\sum_{j=0}^\infty (mj+r)p(n-mj-r)$; 
\item[(ii)] $\displaystyle s_m(n)= \sum_{j=0}^\infty jp(n-mj)$.
\end{itemize}
\end{theorem}

\begin{proof} We note that the combinatorial proof of this theorem is implicit in the combinatorial proof of \cite[Theorem 1.6]{M23}. We write it here in clearer form since it is used in subsequent proofs. 

(i) Let $\mathcal A_{r,j, m}(n)$ be the set of overpartitions of $n$ with a exactly one part  overlined and only a part equal to $mj+r$ maybe overlined. Clearly $$a_{m,r}(n)= \sum_{j=0}^\infty (mj+r)|\mathcal A_{r,j,m}(n)|.$$ The transformation that removes the overlined part from a partition is a bijection from $\mathcal A_{r,j,m}(n)$ to $\mathcal P(n-mj-r)$. This completes the proof of (i). 

(ii) Let $\mathcal S_{j,m}(n)$ be the set of partitions of $n$ in which part $j$ occurs at least $m$ times. Clearly $$s_{m}(n)= \sum_{j=0}^\infty j|\mathcal S_{j,m}(n)|.$$ The transformation that removes $m$ parts equal to $j$ from a partition is a bijection from $\mathcal S_{a,m}(n)$ to $\mathcal P(n-mj)$. This completes the proof of (ii). 
\end{proof}

\begin{theorem}{\cite[Corollary 4.2]{M23}} \label{Cor4.2}Let $m,n$, and $r$ be nonnegative integers such that $0\leq r<m$. Then \begin{equation} \label{C4.2} a_{m,r}(n)+2\sum_{j=1}^\infty (-1)^j a_{m,r}(n-j^2)= \sum_{j=0}^\infty (mj+r)p_{e-o}(n-mj-r).\end{equation}
\end{theorem}

\begin{proof} In \cite{A72},  Andrews constructed an involution and proved  combinatorially that \begin{equation}\label{A1} \sum_{n=-\infty}^\infty (-1)^nq^{n^2}=\sum_{n=0}^\infty \overline p_{e-o}(n)q^n.\end{equation} We define
$$\mathcal P \overline{\mathcal P}(n):= \{(\alpha, \beta)\vdash n \mid \alpha \in \mathcal P, \beta\in \overline{\mathcal P}\}.$$ Andrews' proof together with the combinatorial proof of   Theorem \ref{T1.3} shows combinatorially that  that 
\begin{align*}a_{m,r}(n)&+2\sum_{j=1}^\infty (-1)^j a_{m,r}(n-j^2)\\ & =\sum_{j=0}^\infty (mj+r)\left(|\{(\alpha, \beta)\in  \mathcal P \overline{\mathcal P}(n-mj-r) \mid \ell(\beta) \text{ even}\}|\right.\\  &\qquad \qquad\qquad-\left. |\{(\alpha, \beta)\in  \mathcal P \overline{\mathcal P}(n-mj-r) \mid \ell(\beta) \text{ odd}\}|\right).\end{align*}

Let $n\geq 0$ and define the set $$\mathcal{PQ}(n):=\{(\alpha, \eta)\vdash n \mid \alpha \in \mathcal P, \eta \in \mathcal Q\}.$$  If $n\geq 1$, we define an involution $\psi$ on $\mathcal{PQ}(n)$ by  $$\psi(\alpha, \eta):=\begin{cases}(\alpha\setminus (\alpha_1), \eta\cup (\alpha_1) & \text{ if } \alpha_1>\eta_1, \\ (\alpha\cup (\eta_1), \eta\setminus  (\eta_1) & \text{ if } \alpha_1\leq\eta_1.\end{cases}$$ Clearly, $\psi(\alpha, \eta)$ reverses the parity of $\ell(\eta)$. Thus, if $n\geq 1$,  $$|\{(\alpha, \eta) \in \mathcal{PQ}(n)\mid  \ell(\eta) \text{ even}\}|-|\{(\alpha, \eta) \in \mathcal{PQ}(n)\mid  \ell(\eta) \text{ odd}\}|=0,$$ and if $n=0$ the difference is $1$ since $\{(\alpha, \eta) \in \mathcal{PQ}(0)\}= \{(\emptyset, \emptyset)\}.$

To prove combinatorially that \eqref{C4.2} holds, write $(\alpha, \beta)\in \mathcal P \overline{\mathcal P}(n)$ as $(\alpha, \overline\beta, \widetilde \beta)$, where $\overline \beta$ is the  partition consisting of the overlined parts of $\beta$ and $\widetilde \beta$ is the  partition consisting of the nonoverlined parts of $\beta$. 
Fix $m, n, j,r$ nonnegative integers  such that $0\leq r<m$. 
Furthermore fix a partition $\widetilde \beta$. Define $$\mathcal P \overline{\mathcal P}_{\widetilde \beta}(n-mj-r):=\{(\alpha, \beta)=(\alpha, \overline \beta,\widetilde \beta) \in \mathcal P \overline{\mathcal P}(n-mj-r)\},$$ the set of pairs $(\alpha, \beta)$ in $\mathcal P \overline{\mathcal P}(n-mj-r)$ such that the nonoverlined parts in the overpartition $\beta$ are precisely the parts of $\widetilde \beta$. Using the involution $\psi$ on $(\alpha, \overline \beta)$ (with the overlines removed), we see that \begin{align*}|\{(\alpha, \overline \beta,\widetilde \beta) \in \mathcal P \overline{\mathcal P}_{\widetilde \beta}(n-mj-r)\mid \ell(\overline \beta) \text{ even}\}|& -|\{(\alpha, \overline \beta,\widetilde \beta) \in \mathcal P \overline{\mathcal P}_{\widetilde \beta}(n-mj-r)\mid \ell(\overline \beta) \text{ odd}\}|\\ & = |\{(\emptyset, \emptyset,\widetilde \beta) \in \mathcal P \overline{\mathcal P}(n-mj-r)\}|=1.\end{align*} Summing after all $\widetilde \beta$, we obtain
\begin{align*}&  |\{(\alpha, \beta)\in  \mathcal P \overline{\mathcal P}(n-mj-r) \mid \ell(\beta) \text{ even}\}|- |\{(\alpha, \beta)\in  \mathcal P \overline{\mathcal P}(n-mj-r) \mid \ell(\beta) \text{ odd}\}|\\ &  \  =|\{(\emptyset, \emptyset,\widetilde \beta) \in \mathcal P \overline{\mathcal P}(n-mj-r)\mid\ell(\widetilde \beta) \text{ even}\}|- |\{(\emptyset, \emptyset,\widetilde \beta) \in \mathcal P \overline{\mathcal P}(n-mj-r)\mid\ell(\widetilde \beta) \text{ even}\}|\\ & \  =p_{e-o}(n-mj-r).\end{align*}
\end{proof}

\begin{theorem}{\cite[Corollary 4.4 (i)]{M23}} \label{Cor4.4} Let $m,n$ be nonnegative integers. Then \begin{equation} \label{C4.4} s_{m}(n)+2\sum_{j=1}^\infty (-1)^j s_{m}(n-j^2)= \sum_{j=0}^\infty jp_{e-o}(n-mj).\end{equation}
\end{theorem}

\begin{proof}As in the proof of Theorem \ref{Cor4.2}, \begin{align*}s_{m}(n)&+2\sum_{j=1}^\infty (-1)^j s_{m}(n-j^2)\\ & =\sum_{j=0}^\infty j\left(|\{(\alpha, \beta)\in  \mathcal P \overline{\mathcal P}(n-mj) \mid \ell(\beta) \text{ even}\}|\right.\\  &\qquad \qquad\qquad-\left. |\{(\alpha, \beta)\in  \mathcal P \overline{\mathcal P}(n-mj) \mid \ell(\beta) \text{ odd}\}|\right).\end{align*} Then, the same argument  as in the proof of Theorem \ref{Cor4.2} shows combinatorially that \eqref{C4.4} holds.   
\end{proof}

\begin{theorem}{\cite[Corollary 4.6]{M23}} \label{Cor4.6} Let $m,n$, and $r$ be nonnegative integers such that $0\leq r<m$. Then \begin{equation} \label{C4.6} a_{m,r}(n)+2\sum_{j=1}^\infty (-1)^j a_{m,r}(n-2j^2)= \sum_{j=0}^\infty (mj+r)q_{odd}(n-mj-r).\end{equation}\end{theorem}
\begin{proof}Doubling all parts in Andrews' proof of \eqref{A1} and using Theorem \ref{T1.3}, we have  \begin{align*}a_{m,r}(n)&+2\sum_{j=1}^\infty (-1)^j a_{m,r}(n-2j^2)\\ & =\sum_{j=0}^\infty (mj+r)\left(|\{(\alpha, \beta)\in  \mathcal P \overline{\mathcal P}(n-mj-r) \mid \beta \text{ has even parts, } \ell(\beta) \text{ even}\}|\right.\\  &\qquad \qquad\ \ \ \  -\left. |\{(\alpha, \beta)\in  \mathcal P \overline{\mathcal P}(n-mj-r) \mid \beta \text{ has even parts, } \ell(\beta)  \text{ odd}\}|\right).\end{align*} 

In \cite{G75}, Gupta constructed an involution and gave a combinatorial proof for $$q_{odd}(n)=p_e(n,2)-p_o(n,2),$$ where $p_e(n,2)$ (respectively $p_o(n,2)$) is the number of partitions of $n$ with and even (respectively odd) number of even parts. 

Doubling parts in the argument of the proof of Theorem \ref{Cor4.2}, we see that for $n\geq 1$ \begin{align*}|\{(\alpha, \eta) \in \mathcal{PQ}(n)& \mid  \alpha, \eta \text{ have even parts, }\ell(\eta) \text{ even}\}|\\ & -|\{(\alpha, \eta) \in \mathcal{PQ}(n)\mid \alpha, \eta \text{ have even parts, } \ell(\eta) \text{ odd}\}|=0,\end{align*} and if $n=0$ the difference is $1$ since $\{(\alpha, \eta) \in \mathcal{PQ}(0)\mid \alpha, \eta \text{ have even parts}\}= \{(\emptyset, \emptyset)\}.$

Next, we write $(\alpha, \beta)\in \mathcal P \overline{\mathcal P}(n)$ where $\beta$ has only even parts as $(\alpha^o, \alpha^e, \overline \beta, \widetilde \beta)$, where $\alpha^e$ (respectively $\alpha^o$) is the partition consisting of the even (respectively odd) parts of $\alpha$. Fix $m, n, j, r $ nonnegative integers  such that $0\leq r<m$. 
Furthermore fix a partition $\widetilde \beta$ with even parts and a partition $\alpha^o$ with odd parts.  The involution $\psi$ of Theorem \ref{Cor4.2} is well defined when restricted to pairs of partitions with even parts in $\mathcal{PQ}(n)$.
Using the involution $\psi$ on $(\alpha^e, \overline \beta)$, we obtain \begin{align*}|\{(\alpha^o, \alpha^e, \overline \beta,\widetilde \beta) & \in \mathcal P \overline{\mathcal P}_{\widetilde \beta}(n-mj-r)\mid \ell(\overline \beta) \text{ even}\}|\\&  -|\{(\alpha^o, \alpha^e, \overline \beta,\widetilde \beta) \in \mathcal P \overline{\mathcal P}_{\widetilde \beta}(n-mj-r)\mid \ell(\overline \beta) \text{ odd}\}|\\ & = |\{( \alpha^o,\emptyset, \emptyset,\widetilde \beta) \in \mathcal P \overline{\mathcal P}(n-mj-r)\}|=1.\end{align*}

Summing after all $\alpha^o, \widetilde \beta$, we obtain
\begin{align*}& |\{(\alpha, \beta) \in  \mathcal P \overline{\mathcal P}(n-mj-r) \mid \beta \text{ has even parts, } \ell(\beta) \text{ even}\}|\\   & \qquad-  |\{(\alpha, \beta)\in  \mathcal P \overline{\mathcal P}(n-mj-r) \mid \beta \text{ has even parts, } \ell(\beta)  \text{ odd}\}|\\ \hspace*{-1cm} & = |\{(\alpha^o, \widetilde\beta)\vdash n-mj-r \mid \alpha^o\text{ has odd parts, } \widetilde\beta \text{ has even parts, } \ell(\widetilde\beta) \text{ even}\}|\\  & \qquad-   |\{(\alpha^o, \widetilde\beta)\vdash n-mj-r \mid \alpha^o\text{ has odd parts, } \widetilde\beta \text{ has even parts, } \ell(\widetilde\beta) \text{ odd}\}|\\ & = q_{odd}(n-mj-r).
\end{align*} The last equality above follows from Gupta's involution. 
\end{proof}

\begin{theorem}{\cite[Corollary 4.7(i)]{M23}} Let $m,n$ be nonnegative integers. Then \begin{equation*} \label{C4.7} s_{m}(n)+2\sum_{j=1}^\infty (-1)^j s_{m}(n-2j^2)= \sum_{j=0}^\infty jq_{odd}(n-mj-r).\end{equation*}\end{theorem}
\begin{proof}The proof is identical to the proof of Theorem \ref{Cor4.6} after interpreting the lefthand site as in Theorem \ref{Cor4.4} with the parts of overpartitions doubled. 
\end{proof}

\begin{theorem}{\cite[Corollary 4.9]{M23}} \label{Cor4.9} Let $m,n$, and $r$ be nonnegative integers such that $0\leq r<m$. Then \begin{equation} \label{C4.9} \sum_{j=1}^\infty (-1)^{j(j+1)/2} a_{m,r}(n-j(j+1)/2)= \sum_{j=0}^\infty (mj+r)q\left(\frac{n-mj-r}{2}\right).\end{equation}\end{theorem}
\begin{proof}In \cite{A72},  Andrews constructed an involutions and proved  combinatorially that $$\sum_{n=-\infty}^\infty (-1)^{j(j+1)/2}q^{j(j+1)/2}=\sum_{n=0}^\infty ped_{e-o}(n)q^n,$$ where $ped_{e-o}(n):=ped_e(n)-ped_o(n)$ and $ped_e(n)$ (respectively $ped_o(n)$) is the number of partitions of $n$ with even parts distinct and odd parts unrestricted and an even (respectively odd) total number of parts. Denote by $\mathcal {PED}(n)$ the set of partitions of $n$ with even parts distinct and odd parts unrestricted and let $$\mathcal{PPED}(n):=\{(\lambda, \mu)\vdash n\mid \lambda\in \mathcal P, \mu \in \mathcal{PED}\}.$$ 
Then, Andrews' proof together with the combinatorial proof of Theorem \ref{T1.3} shows combinatorially that 
\begin{align*}\sum_{j=1}^\infty & (-1)^{j(j+1)/2} a_{m,r}(n-j(j+1)/2)\\ & =\sum_{j=0}^\infty (mj+r)\left(|\{(\lambda, \mu)\in  \mathcal{PPED}(n-mj-r) \mid \ell(\mu) \text{ even}\}|\right.\\  &\qquad \qquad\qquad-\left. |\{(\lambda, \mu)\in  \mathcal{PPED}(n-mj-r) \mid \ell(\mu) \text{ odd}\}|\right).
\end{align*} 
We write  $(\lambda, \mu)\in \mathcal{PPED}(n)$ as $(\lambda^e, \lambda^o, \mu^e, \mu^o)$, where $\lambda^e$ (respectively $\lambda^o$) is the partition consisting of the even (respectively odd) parts of $\lambda$; and $\mu^e$, $\mu^o$ are defined similarly. Note that $\mu^e$ is a partition with distinct even parts. 

Fix $m, n, j, r$ nonnegative integers  such that $0\leq r<m$. 
Furthermore fix $\lambda^o$ and $\mu^o$ two partitions with odd parts. Define $$\mathcal{PPED}_{\lambda^o, \mu^o}(n-mj-r):=\{(\lambda, \mu)=(\lambda^e, \lambda^o, \mu^e, \mu^o)\in \mathcal{PPED}(n-mj-r)\}.$$
Using the involution $\psi$ of the proof of Theorem \ref{Cor4.2} on $(\lambda^e, \mu^e)$, we obtain \begin{align*}|\{(\lambda^e, \lambda^o,  \mu^e, \mu^o) & \in\mathcal{PPED}_{\lambda^o,\mu^o}(n-mj-r)\mid \ell(\mu^e) \text{ even}\}|\\&  -|\{(\lambda^e, \lambda^o,  \mu^e, \mu^o)  \in\mathcal{PPED}_{\lambda^o, \mu^o}(n-mj-r)\mid \ell(\mu^e) \text{ odd}\}|\\ & = |\{( \emptyset,\lambda^o, \emptyset,\mu^o) \in \mathcal{PPED}_{\lambda^o, \mu^o}(n-mj-r)\}|=1.\end{align*}
Summing after all $\lambda^o,\mu^o$, we obtain \begin{align*}|\{(\lambda, \mu)& \in  \mathcal{PPED}(n-mj-r) \mid \ell(\mu) \text{ even}\}|\\  &\qquad \qquad\qquad- |\{(\lambda, \mu)\in  \mathcal{PPED}(n-mj-r) \mid \ell(\mu) \text{ odd}\}|\\ & |\{(\lambda^o, \mu^o)\vdash n-mj-r \mid \lambda^o, \mu^o \text{ have odd parts, } \ell(\mu^o) \text{ even}\}|\\ & \qquad \qquad\qquad- |\{(\lambda^o, \mu^o)\vdash n-mj-r \mid \lambda^o, \mu^o \text{ have odd parts, } \ell(\mu^o) \text{ odd}\}|
\end{align*}

In \cite[Proposition 4]{BW}, Ballantine and Welch proved that \begin{align*}|\{(\alpha, \beta)\vdash n &  \mid \alpha, \beta \text{ have odd parts, } \ell(\beta) \text{ even}\}|\\ & -|\{(\alpha, \beta)\vdash n \mid \alpha, \beta \text{ have odd parts, } \ell(\beta) \text{ odd}\}| \\ & = q_{even}(n), \end{align*} where $q_{even}(n)$ is the number of distinct partitions with even parts.

Thus, the left hand side of \eqref{C4.9} equals $$\sum_{j=0}^\infty (mj+r)q_{even}(n-mj-r).$$ Finally, the transformation that maps a partition $\lambda\vdash n$ with distinct even parts to the partition $\mu\vdash n/2$ with $\mu_i=\lambda_i/2$ for all $1\leq i\leq \ell(\lambda)$ is a bijection and shows that $q_{even}(n)=q(n/2)$. This completes the proof. \end{proof}

\begin{theorem}{\cite[Corollary 4.10]{M23}} \label{Cor4.10} Let $m, n$ be nonnegative integers. Then \begin{equation} \label{C4.10} \sum_{j=1}^\infty (-1)^{j(j+1)/2} s_{m}(n-j(j+1)/2)= \sum_{j=0}^\infty jq\left(\frac{n-mj}{2}\right).\end{equation}
\end{theorem}
\begin{proof}The proof is identical to that of Theorem \ref{Cor4.9} after the suitable interpretation of the left hand side.
\end{proof}
We remark a slight error in  \cite[Corollary 2.13]{MS} whose statement is that same as \eqref{C4.10} if $m\mid n$. However,  in \cite[Corollary 2.13]{MS},  the righthand side of \eqref{C4.10} is set to $0$ if $m\nmid n$. For example,  it is easily verified that for $n=10$ and $m=3$ the lefthand side of \eqref{C4.10} equals $2$. 
\smallskip 

The rank $r(\lambda)$ of a partition $\lambda$ is defined \cite{D} as the largest part of $\lambda$ minus the number of parts in $\lambda$. Thus, $r(\lambda)=\lambda_1-\ell(\lambda)$. Let \begin{align*}\mathcal N(n)&:=\{\lambda \vdash n \mid r(\lambda)\geq 0\}\},\\ \mathcal R(n)&:=\{\lambda \vdash n \mid r(\lambda)> 0\}\},
\end{align*} and define $N(n)=|\mathcal N(n)|$ and $R(n)=|\mathcal R(n)|$.
\begin{theorem}{\cite[Corollary 5.2]{M23}} \label{Cor5.2} Let $m,n$, and $r$ be nonnegative integers such that $0\leq r<m$. Then 
\begin{itemize}\item[(i)] $\displaystyle \sum_{j=0}^\infty (-1)^{j} a_{m,r}(n-j(3j+1)/2)= \sum_{j=0}^\infty (mj+r)N(n-mj-r)$
\item[(ii)] $\displaystyle \sum_{j=1}^\infty (-1)^{j+1} a_{m,r}(n-j(3j+1)/2)= \sum_{j=0}^\infty (mj+r)R(n-mj-r)$.
\end{itemize}
\end{theorem}

\begin{proof}Using Theorem \ref{T1.3} shows that it is enough to prove combinatorially that for each $k\geq 0$ we have 

\begin{itemize}\item[(i$^*$)] $\displaystyle \sum_{j=0}^\infty (-1)^{j} p(n-mk-r-j(3j+1)/2)= N(n-mk-r)$
\item[(ii$^*$)] $\displaystyle \sum_{j=1}^\infty (-1)^{j+1} p(n-mk-r-j(3j+1)/2)= R(n-mk-r)$.
\end{itemize}
Since $p(n-mk-r)-N(n-mk-r)=|\{\lambda\vdash n \mid r(\lambda)<0\}$ and, by conjugation, the number of partitions of $n-mk-r$ with negative rank equals the number of partitions of $n-mk-r$ with positive rank, statements (i$^*$) and (ii$^*$) are equivalent. 

For $j \in \mathbb Z$, set $a(j):=j(3j+1)/2$.  In \cite{BZ85}, Bressoud and Zeilberger constructed an involution $$\varphi_{BZ}: \bigcup_{j\in 2\mathbb Z}\mathcal P\left(n-a(j)\right)\to  \bigcup_{j\in 2\mathbb Z+1}\mathcal P\left(n-a(j)\right)$$ as follows. Let $\lambda \in \mathcal P(n-a(j))$ and define\ $\varphi_{BZ}(\lambda)$ to be 
\begin{align*}(\ell(\lambda)+3j-1, \lambda_1-1, \ldots, \lambda_{\ell(\lambda)}-1)\in \mathcal P(n-a(j-1)) & \text{ if } \ell(\lambda)+3j\geq \lambda_1, \\ \ \\  (\lambda_2+1,\ldots, \lambda_{\ell(\lambda)}+1, 1^{\lambda_1-3j-\ell(\lambda) -1})\in \mathcal P(n-a(j+1))  & \text{ if } \ell(\lambda)+3j< \lambda_1,\end{align*} where $1^i$ means that there  are $i$ parts equal to $1$ in the partition. 
 Restricting $\varphi_{BZ}$ we obtain an involution $$\varphi_{BZ}: \mathcal R(n)\, \cup \bigcup_{j\geq 2 \text{ even}}\mathcal P\left(n-a(j)\right)\to  \bigcup_{j\geq 1 \text{ even}}\mathcal P\left(n-a(j)\right).$$ This completes the combinatorial proof of the theorem. 
 \end{proof}

\begin{theorem}{\cite[Corollary 5.3]{M23}} \label{Cor5.3} Let $m,n$  be nonnegative integers. Then 
\begin{itemize}\item[(i)] $\displaystyle \sum_{j=0}^\infty (-1)^{j} s_{m}(n-j(3j+1)/2)= \sum_{j=0}^\infty jN(n-mj)$
\item[(ii)] $\displaystyle \sum_{j=1}^\infty (-1)^{j+1} s_{m}(n-j(3j+1)/2)= \sum_{j=0}^\infty jR(n-mj)$.
\end{itemize}
\end{theorem}
\begin{proof} The proof is identical to that of Theorem \ref{Cor5.2}.
\end{proof}

Garden of Eden partitions we introduced by Hopkins and Sellers in \cite {HS} in connection to the game \textit{Bulgarian solitaire}. They are partition $\lambda$ with all parts less than $\ell(\lambda)-1$. Hence they are precisely the partitions with rank at most $-2$. Denote by $G(n)$ the number of Garden of Eden partitions of $n$. 

\begin{theorem}{\cite[Corollaries 6.2 and 6.3]{M23}} \label{Cor6.2} Let $m,n$, and $r$ be nonnegative integers such that $0\leq r<m$. Then \begin{itemize}\item[(i)] $\displaystyle \sum_{j=0}^\infty (-1)^{j+1} a_{m,r}(n-3j(j+1)/2)= \sum_{j=0}^\infty (mj+r)G(n-mj-r)$
\item[(ii)] $\displaystyle \sum_{j=1}^\infty (-1)^{j+1} s_{m}(n-3j(j+1)/2)= \sum_{j=0}^\infty jG(n-mj)$.
\end{itemize}\end{theorem}

\begin{proof}In \cite{HS}, Hopkins and Sellers give an involutions similar to Bressoud and Zeilberger's involution $\varphi_{BZ}$ described in the proof of Theorem \ref{Cor5.2}, and prove combinatorially that $$G(n)=\sum_{j\geq 1}(-1)^{j+1}p(n-3j(j+1)/2).$$ Together with Theorem \ref{T1.3}, this completes the combinatorial proof of Theorem \ref{Cor6.2}. 
\end{proof}

Given a partitions $\lambda$, we denote by $m_\lambda(1)$  the number of parts equal to $1$ in $\lambda$ and by $w(\lambda)$  the number of parts greater than $m_\lambda(1)$ in $\lambda$. Then the crank $cr(\lambda)$ of $\lambda$ is defined \cite{AG} as $$cr(\lambda):=\begin{cases}\lambda_1 & \text{ if } m_\lambda(1)=0,\\ w(\lambda)-m_\lambda(1) & \text{ if } m_\lambda(1)>0.\end{cases}$$ Let \begin{align*}\mathcal C(n)&:=\{\lambda \vdash n \mid cr(\lambda)\geq 0\}\},\\ \mathcal D(n)&:=\{\lambda \vdash n \mid cr(\lambda)> 0\}\},
\end{align*} and define $C(n)=|\mathcal C(n)|$ and $D(n)=|\mathcal D(n)|$.

\begin{theorem}{\cite[Corollary 7.2]{M23}} \label{Cor7.2} Let $m,n$, and $r$ be nonnegative integers such that $0\leq r<m$. Then 
\begin{itemize}\item[(i)] $\displaystyle \sum_{j=0}^\infty (-1)^{j} a_{m,r}(n-j(j+1)/2)= \sum_{j=0}^\infty (mj+r)C(n-mj-r)$
\item[(ii)] $\displaystyle \sum_{j=1}^\infty (-1)^{j+1} a_{m,r}(n-j(j+1)/2)= \sum_{j=0}^\infty (mj+r)D(n-mj-r)$.
\end{itemize}
\end{theorem}
\begin{proof} Berkovich and Gravan \cite{BG} proved combinatorially that $D(n)$ is also equal to the number of partitions of $n$ with negative crank. Then, as in the proof of Theorem \ref{Cor5.2},   statements (i) and (ii) of  Theorem \ref{Cor7.2} are equivalent. 

Given a partition $\lambda$, the smallest positive integer that is not a part of $\lambda$ is called the minimal excludant of $\lambda$  and is denoted by $\mex(\lambda)$ (see \cite{GK,AN}). For example, $$\mex(7,7,4,2,1,1)=3.$$
If $n, j$ are nonnegative integers with $0<j(j+1)/2\leq n$, and $\lambda \in \mathcal P(n-j(j+1)/2)$, the transformation that adds parts $1, 2, \ldots, j$ to $\lambda$ is a bijection from $\lambda \in \mathcal P(n-j(j+1)/2)$ to the set of partitions $\lambda \in \mathcal P(n)$ with $\mex(\lambda)>j$. This shows combinatorially that, for $n, j\geq 0$, we have  $$p\left(n-\frac{j(j+1)}{2}\right)-p\left(n-\frac{(j+1)(j+2)}{2}\right)=|\{\lambda \in \mathcal P(n) \mid \mex(\lambda)=j+1\}|.$$ Therefore, we have a combinatorial proof that $$\sum_{j\geq 0} (-1)^jp\left(n-\frac{j(j+1)}{2}\right)=|\{\lambda \in \mathcal P(n) \mid \mex(\lambda) \text{ odd}\}|.$$
Hopkins, Sellers and Yee \cite{HSY}, and also Konan \cite{K}, proved combinatorially that $$C(n)=|\{\lambda \in \mathcal P(n) \mid \mex(\lambda) \text{ odd}\}|.$$ Together with Theorem \ref{T1.3}, this completes the combinatorial proof of Theorem \ref{Cor7.2}.
\end{proof}
\begin{theorem}{\cite[Corollary 7.3]{M23}} \label{Cor7.3} Let $m,n$  be nonnegative integers. Then 
\begin{itemize}\item[(i)] $\displaystyle \sum_{j=0}^\infty (-1)^{j} s_{m}(n-j(j+1)/2)= \sum_{j=0}^\infty jC(n-mj)$
\item[(ii)] $\displaystyle \sum_{j=1}^\infty (-1)^{j+1} s_{m}(n-j(j+1)/2)= \sum_{j=0}^\infty jD(n-mj)$.
\end{itemize}
\end{theorem}
\begin{proof}The proof is identical to that of Theorem \ref{Cor7.2}.
\end{proof}


\begin{thebibliography}{00}




 
\bibitem{A98} 
G. E. Andrews, 
\textit{The Theory of Partitions}, 
Cambridge Mathematical Library, Cambridge University Press, Cambridge, 1998. Reprint of the 1976 original.
 

\bibitem{A72} G. E. Andrews,  \textit{Two theorems of Gauss and allied identities proved arithmetically.} Pacific J. Math. 41 (1972), 563--578.

\bibitem{AG} G. E. Andrews, F. G. Garvan, \textit{Dyson’s crank of a partition}. Bull. Amer. Math. Soc. 18 (1988) 167--171.

\bibitem{AN} G. E. Andrews and D. Newman, \textit{Partitions and the minimal excludant}. Ann. Combin. 23(2) (2019) 249--254.

\bibitem{BW} C. Ballantine and A. Welch, \textit{PED and POD partitions: combinatorial proofs of recurrence relations.} Discrete Math. 346 (2023), no. 3, Paper No. 113259, 20 pp.

\bibitem{BG} A. Berkovich and F. G. Garvan,  \textit{Some observations on Dysons new symmetries of partitions.} J. Comb. Theory A 100 (1), 61--93 (2002)





\bibitem{BZ85} D. M. Bressoud, David and D. Zeilberger,  \textit{Bijecting Euler's partitions-recurrence.} Amer. Math. Monthly 92 (1985), no. 1, 54--55.



\bibitem{D} F. Dyson, \textit{Some guesses in the theory of partitions}. Eureka 8 (1944) 10--15.

\bibitem{GK} P. J. Grabner, A. Knopfmacher, \textit{Analysis of some new partition statistics}. Ramanujan J. 12 (2006) 439--454.

\bibitem{G75}
H.  Gupta,
\textit{Combinatorial proof of a theorem on partitions into an even or odd number of parts}. J. Combinatorial Theory Ser. A 21 (1976), no. 1, 100--103.

\bibitem{HS}  B. Hopkins and J.A. Sellers, \textit{Exact enumeration of Garden of Eden partitions}. Integers: Elec. J. of Comb. Number Th. 7(2) (2007), A19.

\bibitem{HSY} B. Hopkins, Brian, J. A.  Sellers and A. J.  Yee, \textit{Combinatorial perspectives on the crank and mex partition statistics.} 
Electron. J. Combin. 29 (2022), no. 2, Paper No. 2.11, 20 pp. 

\bibitem{K} I. 
Konan,
\textit{A Bijective Proof of a Generalization of the Non-Negative Crank-Odd Mex Identity.}
Electron. J. Combin. 30 (2023), no. 1, Paper No. 1.41. 

\bibitem{MS} P. J. Mahanta and M. P. Saikia, \textit{Refinement of some partition identities of Merca and Yee}. Int. J. Number Theory18 (5) (2022), 1131--1142.


\bibitem{M23} M. Merca, \textit{Linear Inequalities concerning the sum of distinct parts congruent to $r$ modulo $m$ in all the partitions on $n$}. Quaest. Math (2023) 1-23, \\ https://doi.org/10.2989/16073606.2023.2174911

\end{thebibliography}
\end{document}